\pdfoutput=1

\documentclass[11pt]{amsart}
\usepackage{amsmath,amssymb}
\usepackage{graphicx}
\usepackage{amsfonts,amscd,latexsym,bbm,epsfig,epic,eepic,oldgerm,psfrag}

\newcommand{\labell}[1] {\label{#1}}

\newlength{\facewd} \newlength{\faceht}%

\newcommand{\ww} {{\bf w}}
\newcommand{\mm} {{\bf m}}
\newcommand{\PP} {{\mathbb P}}

\newcommand{\Ll}{{\mathcal L}}

\newcommand{\less}{{\smallsetminus}}

\newcommand{\p}{{\partial}}

\newcommand{\om}{{\omega}}
\newcommand{\eps}{{\varepsilon}}

\newcommand{\io}{{\iota}}

\newcommand{\la}{{\lambda}}

\newcommand{\si}{{\sigma}}

\newcommand{\Ee}{{\mathcal E}}

\newcommand{\ts}{\textstyle}

\renewcommand{\Tilde}{\widetilde}

\newcommand{\N}{{\mathbb N}}
\newcommand{\Q}{{\mathbb Q}}
\newcommand{\R}{{\mathbb R}}
\newcommand{\C}{{\mathbb C}}

\newcommand{\Z}{{\mathbb Z}}

\newcommand{\Hh}{{\mathcal H}}

\newcommand{\SSS}{{\smallskip}}

\newcommand{\se} {{\stackrel{s}\hookrightarrow}}

\newtheorem{theorem}{Theorem}[section]
\newtheorem{thm}[theorem]{Theorem}

\newtheorem{cor}[theorem]{Corollary}
\newtheorem{lemma}[theorem]{Lemma}

\newtheorem{prop}[theorem]{Proposition}

\newtheorem{defn}[theorem]{Definition}

\newtheorem{rmk}[theorem]{Remark}

\numberwithin{figure}{section}
\numberwithin{equation}{section}
\numberwithin{table}{section}

\newcommand{\MS}{{\medskip}}

\newcommand{\NI}{{\noindent}}
\newcommand{\vol}{{\rm vol}}

\newcommand{\bl}{{\it bl}}

\begin{document}

 \title{Symplectic embeddings and continued fractions: a survey}
 \author{Dusa McDuff}\thanks{partially supported by NSF grant DMS 0604769.}
\address{Department of Mathematics,
Barnard College, Columbia University, New York, NY 10027-6598, USA.}
\email{dmcduff@barnard.edu}
\keywords{symplectic embedding, symplectic capacity, continued fractions, symplectic ellipsoid, symplectic packing, lattice points in triangles}
\subjclass[2010]{53D05, 32S25, 11J70}
\date{29 August 2009, revised October 3 2009. Notes for the Takagi lectures, June 2009. A related course of lectures was also given at the MSRI Graduate Summer school in August 2009.}
\begin{abstract}
 As has been known since the time of Gromov's Nonsqueezing Theorem, symplectic embedding questions lie at the heart of symplectic geometry.  After surveying some of the most important ways of measuring the size of a symplectic set, these notes
  discuss some recent developments concerning the question of when
 a $4$-dimensional ellipsoid can be symplectically
  embedded in a ball.  This problem turns out to have unexpected relations to 
 the properties of continued fractions and of exceptional curves 
 in blow ups of the complex projective plane. It is also related to questions of lattice packing of planar triangles.
\end{abstract}

\maketitle
\begin{center} 
\end{center}

\section{Overview.}

The standard symplectic structure on $\R^{2n}$ is:
$$
\om_0 = dx_1\wedge dx_2 + dx_3\wedge dx_4 + \dots + dx_{2n-1}\wedge dx_{2n}.
$$
By {\bf Darboux's theorem} every symplectic form is locally diffeomorphic to this one, so it is crucial to understand its properties.
\MS\MS

\NI
Let $B : = B^{2n}(a)\subset \R^{2n}$ be the standard (closed) ball of radius $\sqrt a$ 
 (so $a$ is proportional to a  $2$-dimensional area),
and let $\phi: B\;\se\; \R^{2n}$ be a {\bf symplectic embedding}  (i.e. a smooth embedding onto $\phi(B)$ that preserves $\om_0$).
Since $\om_0^n = n! \,dx_1\wedge\dots\wedge dx_{2n}$ 
is a volume form, every 
symplectic embedding preserves volume.
\MS

\NI
{\bf Gromov's question:}  {\it  What can one say about the images $\phi(B)$ of a symplectic ball?  How are the symplectic and volume preserving cases different?}
\MS

The following result can be easily proven using Moser's homotopy method.
\MS

\NI
{\bf Volume preserving embeddings:} 
{\it if $V\subset \R^{2n}$ is diffeomorphic to a ball and $\vol\, V = \vol\, B$ then there is a volume preserving diffeomorphism $\psi: B\stackrel{\cong}\to V$.}
\MS

\NI
 The analogous statement is not true for symplectic diffeomorphisms --- called {\it symplectomorphisms} for short ---
  since the  boundary of $V$ has symplectic invariants given by the characteristic foliation.\footnote
  {
  The characteristic foliation on a smooth hypersurface $Q$ is spanned by the unique null direction of $\om|_Q$, i.e. by the 
  vectors $v\in TQ$ such that $\om(v,w)=0$ for all $w\in TQ$.}
   Nevertheless one might ask if one can {\bf fully fill} $V$ by a symplectic ball.  This means that for every $\eps>0$ there is a ball $B$ and a symplectic embedding $\phi:B\se V$ such that $\vol\bigr(V\less \phi(B)\bigr) <\eps$. 

\MS\MS

\NI Let $Z(A)$ be the cylinder 
$$
B^2(A) \times \R^{2n-2} = \{(x_1,\dots,x_{2n})\in \R^{2n}: x_1^2 + x_2^2 \le A\}$$
We provide $Z(A)$ with the symplectic structure $\om_0$.  Thus its first two coordinates lie in  symplectically paired directions.  In particular,
$\om_0$ has non zero integral over each disc $B^2(A)\times \{pt\}$.
\MS

\NI {\bf Gromov's Nonsqueezing Theorem:}  {\it There is a symplectic embedding $$
\phi:B^{2n}(a)\se Z(A)
$$
 if and only if $a\le A$.}

\begin{figure}[htbp] 
   \centering
   \includegraphics[width=2in]{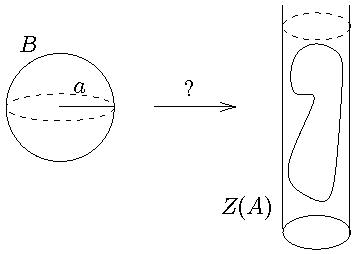} 
   \caption{Does the ball embed symplectically in the cylinder?}
   \label{fig:ex0}
\end{figure}
\MS

This nonsqueezing property is fundamental.  To a first approximation,
it is true that a  diffeomorphism with
this nonsqueezing property for all balls and all symplectic cylinders
must preserve the symplectic structure.  (More precise versions of this statement were proven by Eliashberg \cite{E} and Ekeland--Hofer \cite{EH}.)    

It follows also that symplectomorphisms are very different from volume preserving embeddings.  For example, the largest ball that embeds symplectically in the polydisc $B^2(1)\times B^2(1)$ is $B^{4}(1)$ which has volume just half that of the polydisc. 

\subsection{Symplectic capacities.}

In their paper \cite{EH}, Ekeland and Hofer formalized the idea of a {\it symplectic capacity}, which is a measurement $c(U,\om)$ of the size of a symplectic manifold $(U,\om)$ (possibly open or with boundary) with the following properties:
\MS

\NI (i)  {\it $c$ is a function with values in $[0,\infty]$ defined 
on some class of  $2n$-dimensional symplectic  manifolds\footnote
{
A smooth manifold $M$ is said to be symplectic if it is provided with a symplectic form $\om$.  Here $\om$ is any closed $2$-form such that $\om^n$ never vanishes. A fundamental result due to Darboux is that every symplectic form is locally diffeomorphic to Euclidean space $\R^{2n}$ with its standard structure $\om_0$.  Hence the importance of understanding  the symplectic properties of subsets of $\R^{2n}$.}
$(M,\om)$ that contains the closures of all open subsets of Euclidean space;}

\NI (ii) (monotonicity)  {\it if there is a symplectic embedding  $(U,\om)\to (U',\om')$ then $c(U,\om) \le c(U',\om')$};\SSS

\NI (iii) (scaling) {\it  $c(U,\la\om) =\la c(U,\om)$ for all $\la > 0$};
\SSS

\NI (iv) (normalization) {\it $c\bigl(B^{2n}(a)\bigr) >0$ and 
$c\bigl(Z^{2n}(a)\bigr)< \infty$}.
\MS

The monotonicity property implies that $c$ is a symplectic invariant, that is, it takes the same value on  symplectomorphic sets, while the scaling property implies that it scales  like a $2$-dimensional invariant.   One can satisfy 
  the first three properties  by considering
an appropriate power of the volume; for example one could consider
the function $c(U,\om): = \bigl(\int_U\om^n\bigr)^{1/n}$.    However, this function $c$ does not satisfy  the second half
of the normalization axiom.  Indeed, the requirement that
a cylinder has finite capacity is what makes this an interesting definition.

Here are some examples that illustrate some of the variety of possible definitions.\MS

\NI (i)  {\bf The Gromov width.}   $c_G(U,\om) = \sup\{\pi a\,|\, B^{2n}(a)\se (U,\om)\}.$

The nonsqueezing theorem implies that this does satisfy the axioms.  Moreover, it satisfies the following strong normalization condition:
$$
 c\bigl(B^{2n}(a)\bigr) = c\bigl(Z^{2n}(a)\bigr).
 $$

\MS

\NI (ii)  {\bf The  Hofer--Zehnder capacity \cite{HZ}}.   Every smooth function $H:M\to \R$ on a closed symplectic manifold generates a $1$-parameter subgroup $\phi^H_t, t\in \R$, of the group of symplectomorphisms of 
$(M,\om)$ by the following procedure;  define the vector field  $X_H$, the {\it symplectic gradient} of $H$, by requiring that
$$
\om(X_H,\cdot) = dH(\cdot),
$$
and then define $\phi_t^H$ to be the flow generated by $X_H$ with $\phi_0^H = id$.
The identity
$$
\Ll_{X_H}(\om) = d\bigl(\io(X_H)\om)\bigr) + \io(X_H) d\om = d(dH) = 0,
$$
(where $\Ll$ denotes the Lie derivative)
shows that  $(\phi_t^H)^*\om = \om$ for all $t$. In other words, the flow preserves the symplectic form. 
 A point $x$ for which there are times $0<t<T$ such that
 $\phi_T^H(x) = x$ but  $\phi_t^H(x)\ne x$ is said to be a nontrivial periodic orbit of $\phi_t^H$ of period $T$.  Further 
 we define $\Hh$ to be the set of functions 
 $H:M\to \R$ with the following properties:\MS
 
 $\bullet$  $H(x)\ge 0$ for all $x\in M$, 
 
  $\bullet$ there is an open subset of $M$ on which $H=0$;
  
   $\bullet$ $H$ is constant outside a compact subset of the interior of $M$;
   
    $\bullet$  every  nontrivial periodic orbit of $H$ has period $T\ge 1$.\MS
    
 Then we define the  Hofer--Zehnder capacity $c_{HZ}$ as follows:
 $$
 c_{HZ}(U,\om) = \sup_{H\in \Hh} \Bigl(\sup_{x\in M} H(x)\Bigr)
$$
This function $c_{HZ}$ obviously satisfies the first three conditions for capacities.  Moreover as $\la$ increases the flow of $\la H$  moves faster so that the periods of the periodic orbits decrease.  Therefore in order to prove that a set such as $B^{2n}$ has finite capacity one needs a mechanism to prove that periodic orbits for $\phi_t^H$ must exist under suitable circumstances (for example, if 
$H=\la K$  where $K$ satisfies the first three conditions to be in  $\Hh$ and $\la$ is sufficiently large.)  In the original papers this mechanism  involved subtle arguments in variational analysis; however one can also prove such results using $J$-holomorphic methods.\MS

\NI (iii) {\bf The displacement energy}.    If $H_t: M\to \R, t\in [0,1],$ is a smooth family of functions then one can define a path of vector fields $X_{H_t}$ by requiring that 
$$
\om(X_{H_t},\cdot)= d(H_t)(\cdot) \mbox{ for each }t\in [0,1],
$$
and then integrate this family of vector fields to get a flow $\phi_t^H,t\in [0,1]$.
 As before, each diffeomorphism $\phi_t^H$ preserves $\om$, and the family  is called a Hamiltonian path, or Hamiltonian isotopy.
Define the length $\Ll(\{\phi^H_t\})$ of such a path by setting
$$
\Ll(\{\phi^H_t\}): = \int_0^1\Bigl(\sup_{x\in M} H_t(x) - \inf_{x\in M} H_t(x)\Bigr) dt.
$$
(Thus this is a measure of the total variation of $H_t$.)  

If $U\subset (M,\om)$ then a Hamiltonian isotopy $\phi_t^H$ of $M$ is said to disjoin $U$  in $M$ if $\phi^H_1(U)\cap U=\emptyset.$
  The {\it displacement energy} of a subset $(U,\om)$ of
   a symplectic manifold $(M,\om)$ is defined to be:
$$
d_M(U,\om) = \inf \{\Ll(\{\phi^H_t\}): \{\phi^H_t\} \mbox{ disjoins } U\mbox{ in } M \},
$$
where we take the infimum of an empty set of real numbers to be $\infty$.
  It is a deep fact that the displacement energy of a ball is always positive, no matter what $M$ is.    Using this, it is easy to see that this energy defines a capacity on the set of all subsets of a given symplectic manifold $(M,\om)$. Interestingly enough, as explained in \cite{LM} there are relations between the displacement energy of a ball and the nonsqueezing theorem.

Although the displacement energy of a set $(U,\om)$ in principle
  depends on the choice of ambient manifold $(M,\om)$,
 I am not  aware of any examples of sets $U$ that have
different displacement energies when considered as subsets of two different manifolds $(M_1,\om_1)$ and $(M_2,\om_2)$, except in the trivial case when $U$ is displaceable in $M_1$ but not in $M_2$.   As a refinement of this question one could consider
two different symplectic embeddings  $\io_1$ and $\io_2$ of one set $(U,\om)$ into another $(M,\om)$.    Are there any examples where $d_M(\io_1(U,\om))\ne d_M(\io_2(U,\om))?$  This will not happen if the two embeddings are Hamiltonian isotopic (i.e. there is a Hamiltonian isotopy $\phi_t^H$ of $M$ such that $\io_2=\phi_1^H\circ\io_1$).  But there are cases in which  nonisotopic embeddings are known to exist; in other words, the space of embeddings  $(U,\om)\se (M,\om)$ need not be connected.   Such examples (of embeddings of one polydisc $B^2(a_1)\times B^2(a_2)$ into another)  were constructed   in \cite{FHW}. 
 This leads into the whole question of what is the topology of such embedding spaces. Rather little is known about such questions.\MS

The above three capacities are by now classical invariants.  There are many other 
more recently introduced functions that share some of the properties of capacities. For example, Biran and Cornea \cite[\S6,6]{BC} are interested in exploring the \lq\lq size" of Lagrangian
 submanifolds  $L$ of $(M,\om)$.   A Lagrangian
 submanifold is a submanifold of dimension $n$ on which 
 the symplectic form vanishes identically.     A basic example is the subspace $V_{\R}$ of $(\R^{2n},\om_0)$ spanned by the coordinates $x_1,x_3,\dots,x_{2n-1}$.  Lagrangian submanifolds are fundamental objects of study in symplectic geometry.
 
   Let us write $B_{\R}^{2n}$ for the intersection 
   $B^{2n}(a)\cap V_\R$.  Then the {\it width} $w(L)$ of
 a Lagrangian submanifold  $L\subset (M,\om)$  is defined 
 to be the supremum of $\pi a$ for which there is a symplectic embedding
$$
\phi: \bigl(B^{2n}(a), B_{\R}^{2n}(a)\bigr)\se (M,L)
$$
such that $\phi^{-1}(L) =  B_{\R}^{2n}(a)$.
Although some calculations can be made, 
it is not even known whether every compact 
 Lagrangian submanifold of $\R^{2n}$ must have finite width. 

Because so little is known about these general embedding questions, we shall now return to considering  balls and ellipsoids.
  As we shall see, their embedding properties
   are very closely related, at least in dimension $4$.
 We shall first consider the problem of embedding several equal balls, and then discuss ellipsoids in \S\ref{s:ell}.
\MS

\subsection{Symplectic ball packing.}
\SSS

\NI {\bf Question:}
{\it How much of the volume of $B(1)$ can be filled by $k$ equal symplectic  balls?}\MS

\NI  [Gromov \cite{G}]: \,\,{\it there are obstructions when $k=2$: 

 if $B(a)\sqcup B(a) \se B(1)$ then $a\le 1/2$;}
\MS

\NI in other words, the volume of the image of $B(a)\sqcup B(a)$ is $\le \frac 1{2^{n-1}} \vol\,(B(1)).$\MS

\NI  [McDuff and Polterovich \cite{MP}]: \,\,{\it for any $d\ge 1$ the $2n$-dimensional ball $B(1)$ can be fully filled by $k=d^n$ equal balls.  }\MS

One can explicitly see these embeddings in the following way.
\MS

\NI {\bf Describing embeddings by toric models:}  

 Consider the (moment) map 
$
\Phi:\;\; \C^2\to\R^2, \Phi(z_1,z_2) = (|z_1|^2, |z_2|^2).
$
Then $\Phi(B(1))$ is the (standard) triangle    
$ \bigl\{0\le x_1, x_2;\; x_1+x_2\le 1\bigr\}.
$
Thus the inverse image of the closed triangle can obviously be 
fully  filled by a ball.
 Somewhat surprisingly,  Traynor \cite{Tr} showed that it is possible to fully fill the inverse image of the {\it interior} of any standard triangle.
Hence,  we can obtain  full fillings of $B^4(1)$ by $d^2$  equal balls by cutting a standard triangle into $d^2$ standard pieces as in Figure \ref{fig:jap2}.

\begin{figure}[htbp] 
   \centering
   \includegraphics[width=2.5in]{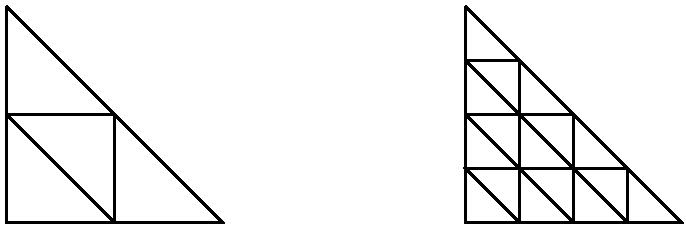} 
   \caption{Embedding $d^2$ standard triangles into a triangle of size $1$.}
   \label{fig:jap2}
\end{figure} \MS

However, when $k\ne d^2$ the story is different:\MS

\NI
 [McDuff--Polterovich \cite{MP}]: {\it In dimension $4$  there are obstructions to full fillings by $k$ balls when} $k<9, k\ne 1,4$; \SSS

\NI
[Biran \cite{B}]  {\it $B^4(1)$ can be fully filled by $k$ equal balls for all $k\ge 9$.}
\MS\MS

Thus, with enough balls, the obstructions disappear.
In later work \cite{B2}, Biran showed that for all (closed) rational\footnote
{$(M,\om)$ is called {\it rational} if  $[a]\in H^2(M;\Q)$.}
 symplectic $4$-manifolds $(M,\om)$ 
there is an integer $N=N(M,\om)$ such that $(M,\om)$ can be fully filled by $k$ balls for all $k\ge N$.
This is rather surprising.  Note that $N$ depends on $\om$ and can diverge to infinity even when $M$ is fixed.
For example, in the case when $M=S^2\times S^2$ and $\om_k = \frac 1k pr_1^*\si \times pr_2^*\si$ for some area form $\si$ on $S^2$ of total area $\pi$, the nonsqueezing theorem shows that one cannot embed 
a ball larger than $B^4(\frac 1k)$.  Since the volume of $(M,\om_k)$ is $\frac{ \pi^2} k$ while that of $B^4(\frac 1k)$ is $\frac{\pi^2}{2k^2}$, we find that
$N(M,\om_k) \ge 2k$.
Thus, it is not at all clear whether one can remove
 the rationality condition on $[\om]$.  It is also not clear whether there is an analogous result in higher dimensions.

\MS

\section{Embedding ellipsoids into balls}\labell{s:ell}

From now on we shall work in $4$ (real) dimensions, and
shall write  $E\se B$ to mean that there is a symplectic embedding of $E$ into $B$.
 Define the ellipsoid $E(a,1)$ by setting:  
 $$
 E(a,1) = \{(x_1,\dots,x_4)\in \R^4: \frac{x_1^2+x_2^2} a + x_3^2+x_4^2\le 1\}.
 $$
Consider the embedding capacity function $c$ for $a\ge1$:
 $$
 c(a): = \inf\{ \mu: E(a,1)\se B(\mu)\}.
 $$
 Note that $c(a)\ge \sqrt {a}$  because $\vol\, E(a,1) = \vol\, B(\sqrt a)$.
\MS

In \cite{CHLS}, Cieliebak, Hofer, Latschev and Schlenk describe a wide variety of symplectic embedding problems in quantitative terms, 
 formulating many interesting questions, but  providing relatively few answers. (See also Schlenk \cite{Sch}.)  In particular, they define the function $c(a)$ described above, but could say rather little about its properties.   
The first most significant open question concerned the value of $c(4)$.  This was in fact first calculated by Opshtein \cite{Op}, though he did not make this explicit in his paper.

As we shall see, the function $c(a)$ turns out to be surprisingly interesting.  It is the first of the capacity functions of 
\cite{CHLS} to be calculated.\footnote{In fact, the proofs of the results stated below are not yet quite complete, so that at this stage  they should still be considered conjectural.}
Note also that although the nonsqueezing theorem stated above is valid  in all dimensions, we only understand the embedding of ellipsoids in $4$-dimensions. As is shown by the work of Guth \cite{Gu} and Hind--Kerman \cite{HK}, very interesting new phenomena appear in higher dimensions.

\begin{figure}[htbp] 
   \centering
   \includegraphics[width=4in]{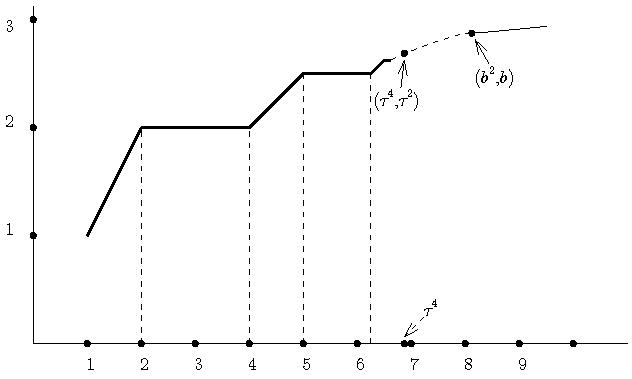} 
   \caption{The approximate graph of $c(a)$. Here $b: = \frac {17}6$.}
   \label{fig:exx}
\end{figure}

\NI
\begin{thm}  [McDuff--Schlenk]\labell{thm:main}    Let $\tau=\frac {1+\sqrt5}{2}$.  The graph of $c(a)$ divides into three parts:\SSS

\NI$\bullet$\,\,  if $1\le a < \tau^4$ the graph is piecewise linear -- an infinite Fibonacci staircase converging to $(\tau^4,\tau^2)$; \SSS

\NI$\bullet$\,\,   $\tau^4\le a< 8\frac 1{36}$  is a transitional region; $c(a)=\sqrt a$ except on a finite number of short intervals;  \SSS

\NI$\bullet$ \,\,if $a\ge 8\frac 1{36} = \bigl(\frac {17}6\bigr)\,\!^2$ then $c(a)=\sqrt a$. 
\end{thm}
\MS

\NI {\bf Description of the  Fibonacci stairs:}   Let $$
g_1=1,\;\; g_2=2, \;\;g_3=5, \;\; g_4=13,\;\; g_5 = 34,\;\; g_6= 89,\dots,
$$
be the odd terms in the sequence of  Fibonacci numbers;  set 
$$
a_n:=\bigr(g_{n+1}/g_n\bigr)\,\!^2,  \;\;b_n:=g_{n+2}/{g_n}$$
so that $a_n < b_n < a_{n+1},$ and $ a_n\to \tau^4.
$  Here we set $g_0:=1$ for convenience, so that
$a_0=1, b_0=2$.
Then the claim is that for all $n\ge 0$ we have
$$
c(x) = x/\sqrt{a_n} \mbox{ on } [a_n, b_n], \;\;\mbox{ and } \;\;
c(x)=\sqrt {a_{n+1}} \mbox{ on }[b_n, a_{n+1}].
$$
Since $\frac {b_n}{\sqrt{a_n}} = \sqrt {a_{n+1}}$ this gives a continuous graph on the interval $1\le a < \tau^4$.
For example, 
$$
c(2) = c(b_0) = \sqrt{a_1}=2 = c(4),\quad 
c(5) = c(b_1) = {\ts \frac 52 = c(6\frac 14)}.
$$
Note  that the function $c(a)/a$ is continuous and nonincreasing.  (This holds because, for all $\la\ge 1$ and $\eps>0$, we have $E(\la a, 1)\subset E(\la a,\la) = \la E(a,1)\subset \la B(c(a) + \eps)$.)  It follows that the function $c$ is determined on the interval $[1,\tau^4]$ by its values at the points $a_n, b_n$.

 \MS
 
As we describe in more detail below, the first and third parts  of Theorem \ref{thm:main} are fully proven, while the proof of the second statement is almost complete.
The proofs  are based on the two Propositions \ref{prop:wgt} and \ref{prop:2} stated below.
The first relates the ellipsoidal embedding problem to a ball packing problem, while the second gives a recipe for solving  ball packing problems.

\begin{prop}[McDuff \cite{M}]\labell{prop:wgt}  For each rational $a\ge 1$ there is a finite weight expansion $\ww(a) = (w_1,\dots,w_M)$ such that $E(a,1)$ embeds symplectically in the interior of $B^4(\mu)$ if and only if
the disjoint union of balls  $\underset{i\le M}\sqcup B(w_i)$ embeds symplectically in the interior of $B^4(\mu)$.
\end{prop}

\begin{figure}[htbp] 
   \centering
   \includegraphics[width=2in]{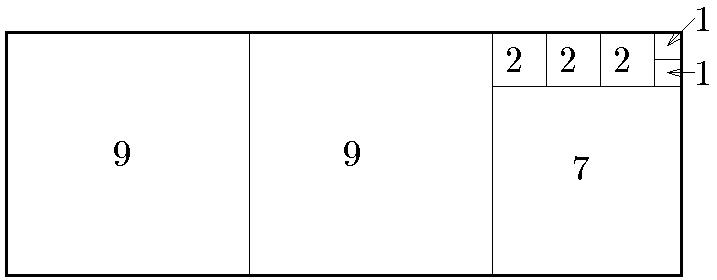} 
   \caption{$\ww(25/9) = (1,\,1,\,7/9,\,2/9,\,2/9,\,2/9,\,1/9,\,1/9)$
   which we abbreviate as $(1^{\times2}, 7/9, (2/9)^{\times3}, (1/9)^{\times2})$.  The {\it multiplicities} $2,1,3,2$ of the weights are the terms (or partial quotients) of the continued fraction  representation $[2;1,3,2]$ of $25/9$.}
   \label{fig:jap3}
\end{figure}

Weight expansions are perhaps best explained pictorially as in Figure \ref{fig:jap3}, but here is a formal definition.

\begin{defn}\labell{def:wa}
Let $a=p/q\in \Q$ written in lowest terms.
The {\bf weight expansion} $\ww: = (w_i): = (w_1,\dots w_k)$ of $a\ge 1$
is defined recursively as follows:
\MS

$\bullet$   $w_1 = 1, $ and $ w_n\ge w_{n+1}>0$ for all $n$;
\SSS
 
$\bullet$  if $ w_i>w_{i+1} = \dots = w_{n}$ (where we set $w_0: = a$)
then
$$
w_{n+1} = \left\{\begin{array} {ll}
 w_{n} &\mbox{if }\;  w_{i+1} + \dots + w_{n+1}=(n-i+1)w_{i+1} \le w_i\\
 w_i - (n-i) w_{i+1} &\mbox{otherwise;} \end{array}\right.
$$

$\bullet$ the sequence stops at $w_n$ if the above formula gives $w_{n+1}=0$.
\SSS

\NI
The number $k$ of nonzero entries in $\ww(a)$ is called the {\bf length} $\ell(a)$ of $a$.
\end{defn}

The second main proposition solves the ball packing problem.
Given $k$ write $\mm: = (m_1,\dots,m_k)\in \N^k$, $m_i\ge m_{i+1}$. Let
\begin{eqnarray*}
\Ee:=\Ee_k:&=& \bigl\{(d;\mm): d^2+1= \sum_im_i^2,\;\; 3d-1=\sum_im_i,\; 
 \mbox{ and  } (*)\bigr\}.
 \end{eqnarray*}
 We shall refer to the first two conditions above as Diophantine conditions.  The third condition $(*)$ is more algebraic,
requiring that tuple $(d;\mm)$ can be reduced to $(0;1)$  
  by repeated Cremona moves.
Such a move takes $(d,\mm)$ to $(d';\mm')$, where $(d';\mm')$ is obtained by first
  transforming $(d;\mm)$ to
$$
(2d-m_1-m_2-m_3; d-m_2-m_3, d-m_1-m_3, d-m_1-m_2, m_4, m_5\dots),
$$ 
and then reordering the new $m_i$ (discarding zeros) so that they do not increase.

It was shown in Li--Li \cite{LL} (using Seiberg--Witten theory) that $\Ee_k$ is
  the set of homology classes $dL-\sum m_iE_i$ represented by symplectic exceptional divisors in the $k$-fold blow up  of $\C P^2$.\footnote
  {
  Here we denote by $L$ the class of a line $[\C P^1]$ and by $E_i$ the class of the $i$th exceptional divisor.  For more about blow ups see \S\ref{ss:proofs}.  Observe also that
  this is a place where symplectic geometry shows how flexible it is  in comparison with algebraic geometry. In algebraic (or complex) geometry, one would want to describe 
 the homology classes that can be represented by holomorphically embedded exceptional divisors for a \lq\lq generic" complex structure on the blow up.  This question is far from being understood.}
  It follows that the intersection number $(d;\mm)\cdot(d';\mm') : = dd'-\sum m_i m_i'$ of any two elements in $\Ee$ is nonnegative. 
  \MS
  
  \begin{prop} [McDuff-Polterovich \cite{MP}, Biran \cite{B}] \labell{prop:2} $$
\underset{i\le k}\sqcup B(w_i)\se B(\mu) \;\Longleftrightarrow a<\mu^2 \mbox{ and } \mu d \ge \sum m_iw_i \ \forall
(d;\mm)\in \Ee_k.
$$
\end{prop}
\begin{cor}
$c(a) =\sup\Bigl\{\sqrt a, \;\underset{(d,\mm)\in \Ee_k}\sup \frac{\sum m_iw_i(a)}d\Bigr\}$.
\end{cor}
\MS

\NI  {\bf Example:}  $\Ee_4 $ has the single element $(1;1,1)$ (corresponding to $L-E_1-E_2$) and 
$\ww(4)=(1,\,1,\,1,\,1) = (1^{\times4})$.  Therefore  $c(4)=2$. However
 $\Ee_5$ also contains $(2;1,\dots,1)= (2; 1^{\times5})$ corresponding to $2L-\sum_{i=1}^5 E_i$.   Thus $c(5) = \sup \{\sqrt 5, 2, 5/2\} = 5/2$.
\MS

The existence of the Fibonacci staircase in Theorem \ref{thm:main} 
is based on the somewhat surprising discovery that 
there are elements of $\Ee$ related to the weight expansions of the odd Fibonacci numbers.  

\begin{prop} [McDuff--Schlenk \cite{MSch}]\labell{prop:fib} 
Denote the odd Fibonacci numbers by $g_n, n\ge 1$, as above, and define $a_n:=(g_{n+1}/g_n)^2$ and  $b_n:=g_{n+2}/g_n$. Then:

\NI
{\rm (i)}
 $
 E(b_n): = \bigl(g_{n+1};g_n\ww(b_n)\bigr)\in \Ee.
 $
 
 \NI {\rm (ii)}  $
E(a_n): = \bigl(g_ng_{n+1}; g_n^2(\ww(a_n)),1\bigr)\in \Ee.
$
  \end{prop}
\begin{proof}[Idea of proof]
It is not hard to show that the elements $E(a_n)$ and $E(b_n)$ satisfy the Diophantine conditions needed to be in $\Ee$.  Moreover, there is an easy inductive proof that $E(b_n)$ satisfies the third condition since  $E(b_n)$ reduces to $E(b_{n-2})$ under five Cremona moves.  However, the corresponding behavior of $E(a_n)$ is much more complicated and takes much more effort to analyze.
\end{proof}

\NI
{\bf Example:}\,\, $n=2$ gives:   
$(5;2\ww(13/2)) = (5; 2^{\times6},\,1^{\times2})$, while 
$$
5L-2(E_1+\dots + E_6) - E_7-E_8\in \Ee_8.
$$
Similarly, $n=4$ gives 
$$
(34; 13 \ww(89/13)) =(34;13^{\times6}, 11, 2^{\times5}, 1^{\times2}),
$$
and one can check that 
$$
34L-13(E_1+\dots + E_6) - 11 E_7 - 2(E_8+\dots+E_{12}) - E_{13}-E_{14} \in \Ee.
$$
\MS

\NI
\begin{cor} For all $n\ge 1$,  $c(a_n) = \sqrt {a_n}$ and
$c(b_n) = \sqrt{a_{n+1}}=c(a_{n+1})$.
\end{cor}
\begin{proof}
Since the intersection of any two elements of $\Ee$ is nonnegative, 
Proposition \ref{prop:fib} part (ii) implies that for any $(d;\mm)\in \Ee$ we have
$$
(d;\mm)\cdot E(a_n) = dg_ng_{n+1} - \sum g_n^2\,m_i w_i(a_n)  - m_{k+1}\ge 0,
$$
where $k: = \ell(a_n)$. Thus
$$
\frac {\mm\cdot\ww(a_n)}d = \frac{\sum m_i w_i(a_n)}{d} 
\le \frac{d\,g_n g_{n+1}} {d\,g_n^2} = \sqrt{a_n}.
$$
It follows that $c(a_n) = \sqrt{a_n}$.  

On the other hand, part (i) of the proposition implies that
\begin{eqnarray*}
d\mu \;\ge\;\sum m_iw_i(b_n) &\Longrightarrow & g_{n+1}\mu\;\ge\; g_n\sum w_i(b_n)^2
\;=\; g_n b_n\\
&\Longrightarrow & \mu\;\ge\;  (g_n g_{n+2})/(g_{n+1}g_n) = \sqrt {a_{n+1}}.
\end{eqnarray*}
This shows $c(b_n)\ge\sqrt{ a_{n+1}}$.  Since $b_n<a_{n+1}$ we must have $c(b_n)\le c(a_{n+1})=\sqrt{ a_{n+1}}$.  Thus
$c(b_n)=c(a_{n+1})$.
\end{proof}

\NI {\bf Claim I:} {\it  $c(a) = \frac {1+a}3$ on the interval $[\tau^4,7]$, with obstruction given by the class $(d;\mm) = (3,2,1^{\times6})\in \Ee_7$.}
\MS

  So far this has been proved for $a>6\frac {11}{12}$ and 
for the (even) convergents to $\tau$.  The proof is quite hard because there are \lq\lq fake" elements of $\Ee$.  More precisely, by looking at ratios of even Fibonacci numbers one can construct tuples $(d;\mm)$ 
that satisfy the Diophantine conditions to be in $\Ee$ but that fail condition $(*)$; in fact  they have negative intersection
with $(3,2,1^{\times 6})$, that is 
$$
3d< 2m_1+ m_2 + \dots +m_7.
$$   There are also infinitely many other elements of $\Ee$, this time constructed from the even Fibonacci numbers,
for which $(\sum m_iw_i(a))/d = \frac {1+a}3$ on some interval $[\tau^4,\tau^4+\eps]$. (These form what might be called a {\it ghost staircase} because it lies below the graph of $c(a)$.)
Thus one cannot understand $c(a)$ on this interval by easy
 estimates using 
the Diophantine conditions.

  On the other hand, this approach is sufficient to show that there are no obstructions on  the interval  $a\ge 8\frac 1{36}$ and that there are finitely many obstructions on  the interval  $7\le a\le 8\frac 1{36}$.  Full details of the proof will be given in \cite{MSch}.
\MS

\NI
{\bf Why there are no constraints for $a>9$.}   
\MS

We saw above that  $c(a) = \sup_{(d,\mm)\in \Ee} \frac{\sum m_iw_i(a)}d$.  Hence
$c(a)>\sqrt a$ only if there is $(d,\mm)\in \Ee$ such that
$ \sum m_iw_i(a)/d>\sqrt a$.   But $\sqrt a\ge 3$, and
$$
\sum m_iw_i \le \sum m_i = 3d-1 \;\;\;\mbox{ since all } w_i\le 1.
$$
So $\sum m_iw_i/d \le 3$ always.  This was Biran's argument in \cite{B}.   
 More complicated versions of this argument also show there are no constraints for $a\ge 8\frac 1{36}.$
\MS

\NI {\bf The special properties of $a=\tau^4$.}\MS

 $a=\tau^4$ is the positive root of $3\sqrt a=a+1$.  The sharpest constraints come  from $(d,\mm)$ where $\mm\approx \la \ww(a)$.  
Since $\sum w_i^2 = a$ and $\sum m_i^2 \approx d^2,$ we need $\la\approx d/\sqrt a$.  But also
$$
 \sum w_i=1+a-1/q\approx 1+a, \;\;\quad
\sum\la w_i \approx \sum_i m_i \approx 3d  
$$
implies $3d\approx \la(1+a)$ or $3\sqrt a \approx 1+a$.
\MS\MS

\subsection{Relation to lattice packing and embedded contact homology.}

The embedded contact homology theory of Hutchings and 
Taubes \cite{HT} is not yet fully developed.  But, according to Hutchings, the index calculations in this theory
 should provide a series of obstructions to symplectically embedding $E(a,1)$ into $B(\mu)$ of the following nature.

 Let $a\ge 1$ be irrational.
For each pair of integers $A,B\ge 0$,  consider the triangle 
$$
T^a_{A,B}: = \bigl\{(x,y)\in \R^2
: x,y\ge 0, x+ay\le A+aB\bigr\}.
$$
We define 
$$
k_{A,B}(a): = {\ts \frac{A+Ba}d,}
$$
 where $d$ is the  smallest positive integer such that
 $$
\#\bigl(T^a_{A,B}\cap \Z^2\bigr) \le {\ts \frac 12} (d+1)(d+2) + k-1,
$$
where $k$ is the number of integer points on the slant edge of $T^a_{A,B}$.
Further, set
$$
c_{ECH}(a): = \sup_{A,B}\bigl\{k_{A,B}(a)\bigr\}.
$$
Since $c_{ECH}(a)$ is nondecreasing, we may extend it to rational $a$ by defining
$$
c_{ECH}(a) = \sup_{z<a, \,z \mbox{ irrat}}\; c_{ECH}(z).
$$

\NI {\bf Claim II:}   $ c_{ECH}(a)\le c(a)$ {\it for all} $a$.\MS

As mentioned above, Claim II is not yet fully proven, but it is expected to hold. Further, it seems probable that these two functions actually coincide.  As evidence for this, we have the following results that were explained to me by Hutchings.

\begin{lemma}\labell{le:ECH1} Suppose that $a$ is rational, define $
T=T^{a}_{A,B}$ as above and suppose that
$$
\# (T\cap \Z^2)\le {\ts \frac 12(d+1)(d+2) + k-1 = \frac 12 (d^2+3d)} + k,
$$
 where $k\ge 1$ is the number of integral points on the slant edge of $T$.  Assume that $(A,B)$ (resp. $(A',B')$) is the integral point on the slant edge with smallest (resp. largest) $x$-coordinate.
  Then there is $\eps>0$ such that 
  $$
c_{ECH}(z) \ge {\ts \frac {A+zB}d} \mbox{ if } z \in (a-\eps,a),\quad
c_{ECH}(z) \ge {\ts \frac {A'+zB'}d} \mbox{ if } z \in (a,a+\eps).
$$
\end{lemma}

\begin{proof}  To prove the statement for $z<a$ it suffices by continuity to consider irrational $z$ of the form $z=a-\eps$.  
Then, for small enough $\eps>0$, the triangle $T^z_{A,B}$ 
contains $k-1$ fewer integral points than  $T$.   Therefore 
$k_{A,B}(z)\ge \frac {A+zB}d$, which proves the first statement.
  Similarly,  the second statement holds because if $z=a+\eps$ is irrational and $\eps>0$ is sufficiently small, the 
triangle $T^z_{A',B'}$ contains $k-1$ fewer integral points than  $T$.
\end{proof}

\begin{lemma}\labell{le:ECH2} For all $n\ge 1$, $c_{ECH}(b_n)\ge \sqrt{a_{n+1}}$.
\end{lemma}
\begin{proof} Consider the triangle $T_n\subset\R^2$ with vertices $(0,0)$, $(g_{n+2},0)$ and $(0,g_n)$, where $g_n$ is the $n$th odd Fibonacci number.  Because $g_n, g_{n+2}$ are mutually prime and satisfy the identities
$$
g_n+g_{n+2} = 3g_{n+1},\qquad g_n g_{n+2} = g_{n+1}^2 + 1,
$$
we find that
\begin{eqnarray*}
\# (T_n\cap \Z^2) &=& {\ts\frac 12} (g_n+1)(g_{n+2}+1) + 1\\
&=&  {\ts\frac 12} (g_{n+1}^2 + 3g_{n+1}) + 2.
\end{eqnarray*}
Therefore, because $b_n = \frac {g_{n+2}}{g_{n}}$,
 we can apply Lemma \ref{le:ECH1} with $k=2$.  Hence  for some $\eps>0$ we have
$$
c_{ECH}(z)\ge{\ts \frac {z g_{n}}{g_{n+1}} \mbox { when } z\in (b_n-\eps, b_n],}
$$
and
$$
c_{ECH}(z)\ge {\ts \frac {g_{n+2}}{g_{n+1}} \mbox { when } z\in [b_n, b_n+\eps).}
$$
In particular, $c_{ECH}(b_n) = \frac {g_{n+2}}{g_{n+1}} = \sqrt{a_{n+1}}$.
\end{proof}

The arguments that prove  Claim II should also prove that $c_{ECH}$ has the same {\it scaling property} as $c$, namely the function
$
\frac {c_{ECH}(a)}{a}$ should be  nonincreasing.
If so, one can
 conclude that $c_{ECH}(a) = c(a)$ for $a\in [1,\tau^4]$. To see this, observe first that  Lemma \ref{le:ECH2} would imply that the two functions agree at $a=b_n$, and second that  $c(a)$ is the smallest nondecreasing function with the scaling property  
that takes the given values at $b_n$.

Numerical evidence suggests that $c_{ECH}(a) = c(a)$ for all $a$.
It seems very likely that, once  $c(a)$ is fully calculated, one could verify this by finding suitable triangles as in Lemma \ref{le:ECH2}.

To my knowledge, rather little is known about lattice counting functions such as $c_{ECH}$. Here is one very simple result.

\begin{lemma}\labell{le:0} 
 $c_{ECH}(a)\ge \sqrt a$ for all $a\ge 1$.
 \end{lemma}
 \begin{proof} This holds because when $s$ is large the number of integral points in 
 $
T^a_{s,0}$ is estimated by its area $\frac 12 s^2/a$ with  error $O(s)$.
Therefore $k_{s,0}\to \sqrt a$ as  $s\to \infty$.  Since $A+Ba= s$ in this case, the result follows.
\end{proof} 

In  \cite{HL}, Hardy and Littlewood\footnote
{
I am indebted to Peter Sarnak for this reference.}
 consider rather different asymptotic questions  about the number of lattice points in $T^a_{A,B}$, looking at asymptotical behavior for fixed $a$ rather than comparing 
the triangles for a given $a$ with those for $a=1$ as we do. Interestingly enough the golden ratio $\tau$ also plays a prominent role in their work. 

\begin{rmk}\rm  Given positive numbers $a,b$ with $a\le b$, denote by $N(a,b)$ the list (with repetitions) of all nonzero numbers of the form $ma + nb$, where $m,n$ are nonnegative integers, put in nondecreasing order.  Thus $N(1,1) = (1,1,2,2,2,3,\dots)$. Write 
$$
N(a,b)\preccurlyeq N(c,d)
$$
 if  for all $k\ge 1$ the $k$ entry of $N(a,b)$ is no greater 
 than the $k$th entry of $N(c,d)$.  Then it is easy to check that
 $$
 c_{ECH}(a) = \inf\bigr\{\mu : N(a,1) \preccurlyeq N(\mu,\mu)\bigr\}.
 $$
 Thus, if $c_{ECH}(a) =c(a)$, it would follow that
 $$
 E(a,1)\se B(\mu,\mu)\; \Longleftrightarrow \;
 N(a,1) \preccurlyeq N(\mu,\mu).
 $$
As pointed out by Hofer,\footnote
{
at a problem session at the Introductory Workshop to the Symplectic Topology program, August 21, 2009  at MSRI, Berkeley.} it is an easy step
from here to wonder whether 
$$
 E(a,1)\se E(c,d)\;  \Longleftrightarrow \;
 N(a,1) \preccurlyeq N(c,d).
 $$
Note that the embedding results of Guth \cite{G} and Hind--Kerman  \cite{HK} imply that the analogous 
result in dimensions $> 4$ does {\it not} hold.  However, embedded contact homology also does not exist in dimensions $> 4$. The question of what the obstructions actually are in higher dimensions is very interesting.
\end{rmk}

\subsection{Ideas behind the proofs.}\labell{ss:proofs}

Finally, we describe the basic ideas 
 behind the proofs of the two key  Propositions \ref{prop:wgt} and \ref{prop:2}.  \MS

\NI {\bf I: Embedding balls and blowing up:}\MS

 To blow up a point in complex geometry you remove the point and replace it by the set of all (complex) 
 lines through that point. In two complex dimensions, this 
 set is a complex line $E\cong \C P^1$ called the {\bf exceptional divisor}.  Its normal bundle $\pi: L\to E$ has Chern class $-1$.   In this context the  basic fact is that
   $L\less E$ is biholomorphic to  $\C ^2\less \{0\}$.
\MS

On the other hand, in symplectic geometry the picture of blowing up is somewhat different.  
 Darboux's theorem shows that a neighborhood of a point $p$ in $(M^4,\om)$ can be identified with a neighborhood  of $\{0\}$ in $\R^4\equiv \C^2$.  Therefore one can
 blow up as before (with respect to a compatible complex structure) to get a manifold $\Tilde M$ with blow down map
$\bl:\Tilde M \to M$.  But one must put  a symplectic form on $\Tilde M$, and the pullback form $\bl^*(\om)$ vanishes on $E$.
  So the symplectic structure $\Tilde\om_a$ on $\Tilde M$ 
must have the form $bl^*(\om) + a\pi^*(\si_E)$ near  $E$, where $\int_E\si_E=1$ and $\pi:$ nbhd$(E)\to E$.
In this context the basic fact is the following:\MS

\begin{center}
 $\bigl(($nbhd$\,E\;\less\;E),\Tilde\om_a\bigr)$ is symplectomorphic to  $\C^2\less B(a)$.
\end{center}
\MS

\NI This means that to blow up symplectically with weight $a$  one removes an embedded copy of the open  ball ${\rm int\,}B(a)$ and then collapses $\p B(a)=S^3$ to $S^2\equiv E$ by the Hopf map (along the leaves of the characteristic foliation). 

Proposition \ref{prop:2} is proved by using this description of symplectic blowing up to convert the ball packing problem into a question about the existence of symplectic forms on the $k$-fold blow up of $\C P^2$.  The latter problem is solved using the theory of $J$-holomorphic curves as in \cite{MP,B}.
\MS\MS

\NI {\bf II: Cutting $E(a,1)$ into balls via toric models} 
\SSS

We now describe the basic ideas behind the proof of Proposition \ref{prop:wgt}.

If $a=p/q$ the image of $E(a,1)$ under the moment map $\Phi: \C^2\to \R^2, (z_1,z_2)\mapsto
(|z_1|^2, |z_2|^2)$  is the triangle 
$$
\{0\le x_1,x_2;\,\, qx_1+px_2\le p\}.
$$
As illustrated on the left of Figure \ref{fig:4}, this  corresponds to a singular variety, with singular points of orders $p,q$.
(In contrast, the triangles in  Figure \ref{fig:jap2} are standard, i.e. equivalent under an integral affine transformation to a triangle with vertices $(0,0), (a,0), (0,a)$, and so represent the smooth variety $\C P^2$.)
We can cut this singular triangle into standard triangles of different sizes.

\begin{figure}[htbp] 
   \centering
   \includegraphics[width=4in]{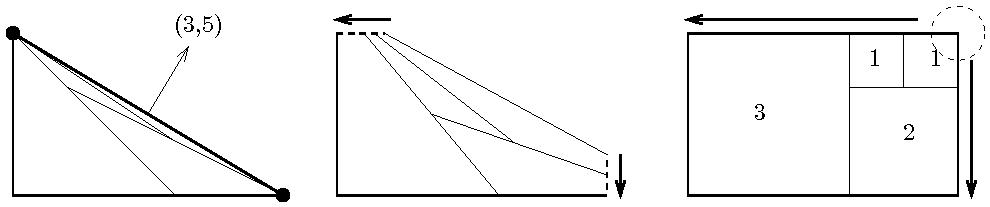} 
\caption{Cutting triangles into standard triangles as compared to cutting a rectangle into squares.}
   \label{fig:4}
\end{figure}

\NI
Figure \ref{fig:4} also shows that the triangle decomposition on the left is structurally the same as that giving the weights for $5/3$. One needs to remove the top right corner of the rectangle and then collapse its top and right side.  As we explain in more detail in \cite{M}, it also corresponds to a joint resolution of the
two singularities of the toric variety corresponding to the {\it complement} of the triangle in the positive quadrant.

\begin{figure}[htbp] 
   \centering
   \includegraphics[width=5in]{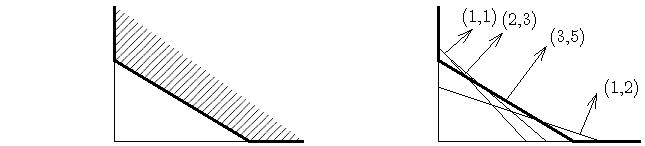} 
\caption{The lines on the right hand figure are the cuts needed to resolve the singularities of the complement of the triangle in the left hand figure.}
   \label{fig:5}
\end{figure}

Figure \ref{fig:5} illustrates another key point.  
The cuts needed to resolve the two singularities in the shaded polytope (the result of removing the singular triangle) are illustrated on the right.  Note that they are {\it parallel} to the cuts that decompose the singular triangle on the left of Figure \ref{fig:4} into standard triangles.
(These cuts resolve the singularities because the matrices whose rows are adjacent normal vectors all have determinant $\pm 1$.) 
\MS

\begin{rmk}\rm  The information contained in the weight expansion
of $a$ can be summarized in a diagram called the {\it Riemenschneider staircase}.  This usually arises in the context of resolving singularities so that the interesting numbers are the coefficients of the Hirzebruch--Jung continued fraction expansions of the two related fractions $p/q$ and $p/(p-q)$ (where we assume $p>2q$); cf. Fulton \cite{F}.  However one can also construct the diagram from the multiplicities $n_1, \dots, n_S$ of the entries in the weight expansion $\ww(a)$.
 We  illustrate this by the two examples in
Figure \ref{fig:jap5}. See Popescu-Pampu \cite{PP} for an extended discussion of these combinatorics; and Craw--Reid \cite{CR} for an example of the use of this diagram in current algebraic geometry.
 Also compare \cite[Remark~3.12(ii)]{M}.

\begin{figure}[htbp] 
   \centering
   \includegraphics[width=2.5in]{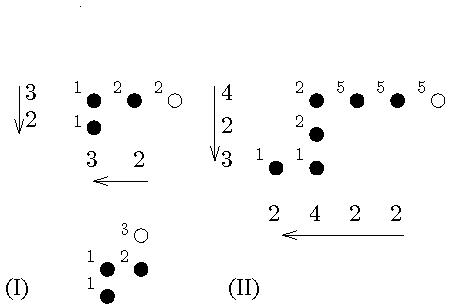} 
   \caption{(I)  illustrates $a=p/q=5/2 = [2;2]$ while (II) illustrates 
$a=p/q=17/5=[3;2,2]$.  The labels $W_i(a)$  are the smaller numbers by the dots; the vertical and horizontal numbers give the components of the Hirzebruch--Jung continued fractions for $p/q>2$ and $p/(p-q)<2$. The arrows give the direction in which one should read the numbers.}
   \label{fig:jap5}
\end{figure}

Observe first that if $p,q$ are relatively prime with $p>q$  then exactly one of the numbers
 $p/q, p/(p-q)$ is greater than two.  Thus, without loss of generality, we may assume that $p/q>2$ so that $n_0\ge2$.  
Given positive integers $n_0,\dots,n_k$ define 
 $$
{\ts \frac pq} = [n_0;n_1,n_2,\dots,n_k] = n_0 + \frac1{n_1 + \frac 1{n_2 + \dots + \frac 1{n_k}}}.
 $$
It is not hard to see that the entries of the weight expansion
$ \ww(\frac pq) = (w_0,\dots,w_N)$ have multipicities $n_0,\dots,n_k$. 
Moreover $q w_i\in \Z$ for all $i$. In other words, the vector 
$W(\frac pq) = (W_0,\dots,W_N):=q\ww(\frac pq)$ has the form
$$
{\ts W(\frac pq)} =  \Bigl(\underbrace{X_0,\dots,X_0}_{n_0}, 
\underbrace{X_1,\dots,X_1}_{n_1}, \dots, \underbrace{X_k,\dots,X_k}_{n_k}\Bigr),
$$ 
where $X_i\in \Z$ and $X_k=1$. We shall call these renormalized weights $W_i=qw_i$ the {\it labels} of $a=\frac pq$.

  The  staircase is constructed from the upper right, proceeding to the left and down, by first placing $n_0$ horizontal dots, then $n_1$ vertical dots (starting in the row of the first dots), then
 $n_2$ horizontal dots (starting in the row of the second dots), and so on.
  The first dot is white, and the others are black.  One can recover the labels from the staircase by starting with $1$ at the bottom left and moving back up the staircase according to the rule:

\begin{itemize}\item[]{\it at a horizontal move, the new label is the sum of the labels in the preceding column;
at a vertical move, the new label is the sum of the labels in the preceding row.}
\end{itemize}
%
%

\NI
The coefficients of the Hirzebruch--Jung continued fraction for $p/q$ are one more than the number of  black dots in each row (read downwards).  If one moves the white dot, putting it in a new row instead of a new column, and constructs the labels as described above,
then one gets the labels for the associated fraction $p/(p-q)$. Correspondingly
 the coefficients of the Hirzebruch--Jung continued fraction for $p/(p-q)$ are one more than the number of  blacks dots in each column (read from right to left).
Thus Figure \ref{fig:jap5} (I) gives $W(\frac 52) = (2,2,1,1)$ and
$$
{\ts \frac 52 = 2 + \frac 12 = 3-\frac{1}{2},\quad
\frac 53 = 1 + \frac{1}{1+\frac 12} = 2-\frac{1}{3};}
$$ 
while Figure \ref{fig:jap5} (II) gives
$W(\frac {17}5) = (5,5,5,2,2,1,1)$ and
$$
{\ts \frac {17}5 = 3 + \frac{1}{2+\frac 12} = 
4 -\frac{1}{2-\frac 13},\quad
\frac {17}{12} = 2 - \frac{1}{2-\frac 1{4-\frac 12}}.}
$$
\end{rmk}

\NI {\bf Acknowledgements}  I thank Felix Schlenk and Dorothee M\"uller for
their comments on an earlier version of this manuscript.

\end{document}